\documentclass[11pt,oneside,onecolumn]{amsart}
\usepackage{amssymb}
\usepackage{xcolor}
\usepackage{mathrsfs}
\usepackage{amsfonts}
\usepackage{amsmath}
\usepackage{tikz}

\numberwithin{equation}{section}
\allowdisplaybreaks[4]
\newtheorem{Theorem}{Theorem}[section]
\newtheorem{nonum}{Theorem}

\newtheorem{Lemma}[Theorem]{Lemma}

\newtheorem{Definition}[Theorem]{Definition}

\unless\ifcsname proof\endcsname

\fi

\def\supp{{\rm supp}\ }
\def\ep{\varepsilon}

\def\bbZ{\mathbb{Z}}

\def\bbR{\mathbb{R}}

\def\Sp{\mathcal{S}}
\def\M{\mathcal{M}}
\def\D{\mathscr{D}}

\begin{document}

\title[Weighted estimates for multilinear operators]{The sharp weighted bound for multilinear maximal functions and Calder\'{o}n-Zygmund operators}


\thanks{The second author is partially supported by the NSF under grant 1201504. The third author is supported partially by the National Natural Science Foundation of China (10990012)  and the Research Fund for the Doctoral Program of Higher Education.}

\author{Kangwei Li}
\email{likangwei9@mail.nankai.edu.cn}

\address{School of Mathematical Sciences and LPMC,  Nankai University,
      Tianjin~300071, China}

\author{Kabe Moen}

\email{kabe.moen@ua.edu}
\address{Department of Mathematics, University of Alabama, Tuscaloosa, AL 35487, USA}

\author{Wenchang Sun}

\email{sunwch@nankai.edu.cn}

\address{School of Mathematical Sciences and LPMC,  Nankai University,
      Tianjin~300071, China}

\begin{abstract}
We investigate the weighted bounds for multilinear maximal functions and Calder\'{o}n-Zygmund operators from $L^{p_1}(w_1)\times\cdots\times L^{p_m}(w_m)$
to $L^{p}(v_{\vec{w}})$, where $1<p_1,\cdots,p_m<\infty$ with $1/{p_1}+\cdots+1/{p_m}=1/p$ and $\vec{w}$ is a multiple $A_{\vec{P}}$ weight.  We prove the sharp bound for the multilinear maximal function for all such $p_1,\ldots, p_m$ and prove the sharp bound for $m$-linear Calder\'on-Zymund operators when $p\geq 1$.
\end{abstract}

\maketitle
\section{Introduction and Main Results}

A weight is a non-negative locally integrable function.  Given $p$, $1<p<\infty$, an $A_p$ weight is one that satisfies the following
$$[w]_{A_p} =\sup_Q \Big(\frac{1}{|Q|}\int_Q w\Big)\Big(\frac{1}{|Q|}\int_Q w^{1-p'}\Big)^{p-1}<\infty.$$
It is well known that the Hardy-Littlewood maximal operator and Calder\'on-Zygmund operators are bounded on $L^p(w)$ when $w\in A_p$.  The sharp dependence for the Hardy-Littlewood maximal function is given by
\begin{equation}\label{eq:Buck}
  \|M\|_{L^p(w)\rightarrow L^p(w)}\le C_{n,p}[w]_{A_p}^{\frac{p'}{p}}.
\end{equation}
Inequality \eqref{eq:Buck} was first proven by Buckley \cite{B}. (We refer the reader to \cite{Ler1} for a beautiful proof of this inequality and a summary of the history.)  Later is was proven by Hyt\"onen \cite{Hy1} that if $T$ is a Calder\'on-Zygmund operator then
\begin{equation}\label{eq:Hy}\|T\|_{L^p(w)\rightarrow L^p(w)}\leq C_{n,p,T}[w]_{A_p}^{\max(1,\frac{p'}{p})}.\end{equation}
Again, we refer the reader to \cite{Hy2,Ler2} for the background material and further references.  In this article we prove the the multilinear analogs of inequalities \eqref{eq:Buck} and \eqref{eq:Hy}.  We begin with a few definitions.

First, let us define multiple $A_{\vec{P}}$ weights.
In \cite{LOPTT}, Lerner, Ombrosi, P\'{e}rez, Torres and Trujillo-Gonz\'{a}lez introduced the theory of multiple $A_{\vec{P}}$ weights.
\begin{Definition}
Let $\vec{P}=(p_1,\cdots,p_m)$ with $1\le p_1,\cdots,p_m<\infty$ and $1/{p_1}+\cdots+1/{p_m}=1/p$. Given $\vec{w}=(w_1,\cdots, w_m)$, set
\[
  v_{\vec{w}}=\prod_{i=1}^m w_i^{p/{p_i}}.
\]
We say that $\vec{w}$ satisfies the multilinear $A_{\vec{P}}$ condition if
\[
  [\vec{w}]_{A_{\vec{P}}}:=\sup_Q \left(\frac{1}{|Q|}\int_Q v_{\vec{w}}\right)\prod_{i=1}^m\left( \frac{1}{|Q|}\int_Q w_i^{1-p_i'}\right)^{p/{p_i'}}<\infty,
\]
where $[\vec{w}]_{A_{\vec{P}}}$ is called the $A_{\vec{P}}$ constant of $\vec{w}$. When $p_i=1$, $( \frac{1}{|Q|}\int_Q w_i^{1-p_i'})^{1/{p_i'}}$ is
understood as $(\inf_Q w_i)^{-1}$.
\end{Definition}
It is easy to see that in the linear case (that is, if $m = 1$) $[\vec{w}]_{A_{\vec{P}}}=[w]_{A_p}$ is the
usual $A_p$ constant. In \cite{LOPTT}, it was shown that for $1<p_1,\cdots,p_m<\infty$, $\vec{w}\in A_{\vec{P}}$ if and only if $w_i^{1-p_i'}\in A_{mp_i'}$ and $v_{\vec{w}}\in A_{mp}$.

Given $\vec{f}=(f_1,\cdots,f_m)$, we define the multilinear maximal function by
\[
  \mathcal{M}(\vec{f})=\sup_{Q\ni x}\prod_{i=1}^m\frac{1}{|Q|}\int_{Q}|f_i|.
\]

In \cite{LOPTT}, the authors proved that $\vec{w}\in A_{\vec{P}}$ if and only if
\[
  \|\mathcal{M}(\vec{f})\|_{L^p(v_{\vec{w}})}\le C \prod_{i=1}^m \|f_i\|_{L^{p_i}(w_i)}.
\]

Recall that inequality \eqref{eq:Buck} is sharp in the sense that the exponent on $[w]_{A_p}$
cannot be improved.  The analogous question for the operator $\M$ has remained open.  In \cite{DLP}, by mixed estimates involving $A_\infty$, Dami\'{a}n, Lerner and P\'{e}rez proved the following result.

\begin{nonum}\cite[Theorem 1.2]{DLP}
Let $1<p_i<\infty$, $i=1,\cdots,m$ and $1/p=1/{p_1}+\cdots+1/{p_m}$. Denote by $\alpha=\alpha(p_1,\cdots,p_m)$ the best possible
power in
\begin{equation}\label{eq:em}
  \|\mathcal{M}(\vec{f})\|_{L^p(v_{\vec{w}})}\le C_{m,n,\vec{P}} [\vec{w}]_{A_{\vec{P}}}^\alpha\prod_{i=1}^m \|f_i\|_{L^{p_i}(w_i)}.
\end{equation}
Then the following results hold:
\begin{enumerate}
\item for all $1<p_1,\cdots,p_m<\infty$, $\frac{m}{mp-1}\le \alpha\le \frac{1}{p}\left(1+\sum_{i=1}^m\frac{1}{p_i-1}\right)$.
\item if $p_1 = p_2 = \cdots= p_m = r > 1$, then $\alpha=\frac{m}{r-1}$.
\end{enumerate}
\end{nonum}
Interestingly, the mixed estimates involving $A_\infty$ do not yield the sharp dependence on the constant $[\vec w]_{A_{\vec{P}}}$.  The sharp bound along the diagonal is obtained using similar methods to those found in \cite{Ler1} and these techniques only seem to work along the diagonal.  In this paper we find the optimal power on $[\vec{w}]_{A_{\vec{P}}}$ for the full range of exponents, $1<p_1,\ldots,p_m<\infty$.  
\begin{Theorem}\label{thm:m}
Suppose $1<p_1,\ldots,p_m<\infty$, $1/p=1/{p_1}+\cdots+1/{p_m}$, and $\vec{w}\in A_{\vec P}$.  Then
\begin{equation}\label{eq:emm}
  \|\mathcal{M}(\vec{f})\|_{L^p(v_{\vec{w}})}\le C_{m,n,\vec{P}} [\vec{w}]_{A_{\vec{P}}}^{\max(\frac{p_1'}{p},\cdots, \frac{p_m'}{p})}\prod_{i=1}^m \|f_i\|_{L^{p_i}(w_i)}.
\end{equation}
Moreover the exponent $[\vec w]_{A_{\vec P}}$ is the best possible.
\end{Theorem}
We emphasize that our bounds \eqref{eq:emm} not only improves those in Theorem A, but is also the best possible.  See Figure \ref{fig1} for a visualization of the bilinear case.  
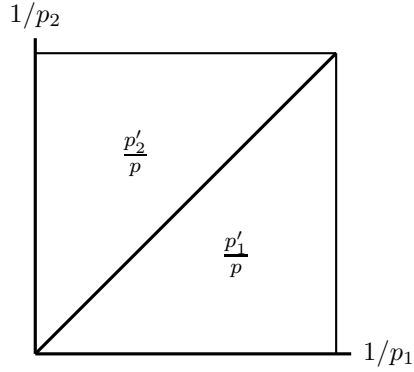
\begin{figure}
\begin{center}\label{fig1}
\begin{tikzpicture}
		\draw[very thick,-] (0,0) -- (4.2,0);
		\draw[very thick,-] (0,0) -- (0,4.2);

		\draw[thick,-] (4,0) -- (4,4);
		\draw[thick,-] (0,4) -- (4,4);
		\draw[very thick,-] (0,0) -- (4,4);

		\draw[] (0,4.5) node {\small$1/p_2$};

		\draw[] (4.7,0) node {\small$1/p_1$};

		\draw[] (2.67,1.33) node {$\frac{p_1'}{p}$};
		\draw[] (1.33,2.67) node {$\frac{p_2'}{p}$};

\end{tikzpicture}

\end{center}
\caption{The sharp exponents on $[\vec w]_{A_{\vec{P}}}$ for the bilinear maximal function.}
\end{figure}

Next we turn to study weighted bounds of multilinear Calder\'{o}n-Zygmund operators.
The theory of  multilinear Calder\'{o}n-Zygmund operators originated in the works of Coifman
and Meyer \cite{CM1,CM2} and was later developed by Christ and Journ\'{e} \cite{CJ},
Kenig and Stein \cite{KS}, and Grafakos and Torres \cite{GT}. The last work provides a comprehensive account of general multilinear Calder\'{o}n-Zygmund operators which we follow in this paper.

\begin{Definition}
Let $T$ be a multilinear operator initially defined on the $m$-fold product of Schwartz spaces and
taking values into the space of tempered distributions,
\[
  T: \mathscr{S}(\bbR^n)\times\cdots\times \mathscr{S}(\bbR^n)\rightarrow \mathscr{S}'(\bbR^n).
\]
we say that $T$ is an $m$-linear Calder\'{o}n-Zygmund operator if, for some $1\le q_i<\infty$, it extends to a bounded multilinear operator
from $L^{q_1} \times\cdots\times L^{q_m}$ to $L^q$ , where $1/{q_1}+\cdots+1/{q_m}=1/q$,
 and if there exists a function $K$, defined off the diagonal $x=y_1=\cdots=y_m$
in $(\bbR^n)^{m+1}$, satisfying
\[
  T(f_1,\cdots,f_m)=\int_{(\bbR^n)^m}K(x,y_1,\cdots,y_m)f_1(y_1)\cdots f_m(y_m)dy_1\cdots dy_m
\]
for all $x\notin \bigcap_{j=1}^m\supp f_i$;
\[
 |K(y_0,y_1,\cdots,y_m)|\le \frac{A}{(\sum_{k,l=0}^m |y_k-y_l|)^{mn}}
\]
and
\[
  |K(y_0,\cdots,y_i,\cdots,y_m)-K(y_0,\cdots,y_i',\cdots,y_m)|\le \frac{A|y_i-y_i'|^\varepsilon}{(\sum_{k,l=0}^m |y_k-y_l|)^{mn+\varepsilon}}
\]
for some $A,\varepsilon>0$ and all $0\le i\le m$, whenever $|y_i-y_i'|\le \frac{1}{2}\max_{0\le k\le m}|y_i-y_k|$.
\end{Definition}

It was shown in \cite{GT} that if $1/{r_1}+\cdots+1/{r_m}=1/r$, then an $m$-linear Calder\'{o}n-Zygmund operator $T$ is bounded from $L^{r_1}\times\cdots\times L^{r_m}$
to $L^r$ when $1<r_i<\infty$ for all $i=1,\cdots,m $; and $T$ is bounded from $L^{r_1}\times\cdots\times L^{r_m}$ to $L^{r,\infty}$ when at least one $r_i=1$. In particular, $T$ is bounded from $L^1\times\cdots\times L^1$ to $L^{1/m,\infty}$.



A weighted theory for $m$-linear Calder\'{o}n-Zygmund operators was developed in \cite{LOPTT},  where it was shown that such operators are bounded from $L^{p_1}(w_1)\times\cdots\times L^{p_m}(w_m)$ to $L^{p}(v_{\vec{w}})$ when $\vec{w}\in A_{\vec{P}}$.  In \cite{DLP}, the authors proved a multilinear version of the $A_2$ conjecture. Specifically, for $p_1=\cdots=p_m=m+1$,
it was shown that
\[
  \|T(\vec{f})\|_{L^p(v_{\vec{w}})}\lesssim [\vec{w}]_{A_{\vec{P}}}\prod_{i=1}^m\|f_i\|_{L^{p_i}(w_i)},
\]
where the estimate for the power of $[\vec{w}]_{A_{\vec{P}}}$ is the best possible.

Due  to the lack of an appropriate extrapolation theorem for multilinear operators with multiple weights, let alone a version with good constants, the sharp estimate  is unknown for any other choices of  $p_i$.
In this paper, we give a sharp estimate for the case of $p\geq 1$. Specifically, we prove the following Theorem, again we refer the reader to Figure 2 for a visualization of the bilinear case.
\begin{Theorem}\label{thm:main}
Let $T$ be a multilnear Calder\'{o}n-Zygmund operator, $\vec{P}=(p_1,\cdots,p_m)$ with $1<p_1,\cdots,p_m<\infty$ and $1/{p_1}+\cdots+1/{p_m}=1/p\leq1$.
If $\vec{w}=(w_1,\cdots, w_m)\in A_{\vec{P}}$, then
\begin{equation}\label{eq:eT}
  \|T(\vec{f})\|_{L^p(v_{\vec{w}})}\le C_{n,m,\vec{P},T}[\vec{w}]_{A_{\vec{P}}}^{\max(1,\frac{p_1'}{p},\ldots,\frac{p_m'}{p})}\prod_{i=1}^m\|f_i\|_{L^{p_i}(w_i)}.
\end{equation}
Moreover, for certain multilinear operators the exponent on $[\vec{w}]_{A_{\vec{P}}}$ is the best possible
\end{Theorem}

\begin{figure}
\begin{center}
\label{fig2}

\begin{tikzpicture}
		\draw[very thick,-] (0,0) -- (4.2,0);
		\draw[very thick,-] (0,0) -- (0,4.2);

		\draw[thick,-] (4,0) -- (4,4);
		\draw[thick,-] (0,4) -- (4,4);
		\draw[very thick,-] (0,4) -- (4,0);

		
		\draw[very thick,black,-] (0,2) -- (4/3,4/3);
		\draw[very thick,black,-] (4/3,4/3) -- (2,0);
		\draw[very thick,black,-] (4/3,4/3) -- (2,2);


	          \filldraw[black] (4/3,4/3) circle (2pt);
		\draw[very thick,-] (0,4) -- (4,0);

		\draw[] (0,4.5) node {\small$1/p_2$};

		\draw[] (4.7,0) node {\small$1/p_1$};
		\draw[] (3,2.5) node {\small{???}};
		\draw[] (.7,.7) node {\small$1$};
		\draw[] (2.5,.8) node {\small$\frac{p_1'}{p}$};
		\draw[] (.8,2.5) node {\small$\frac{p_2'}{p}$};

\end{tikzpicture}

\end{center}
\caption{The sharp exponents on $[\vec w]_{A_{\vec{P}}}$ for bilinear Calder\'on-Zygmund operators when $p\geq 1$}
\end{figure}
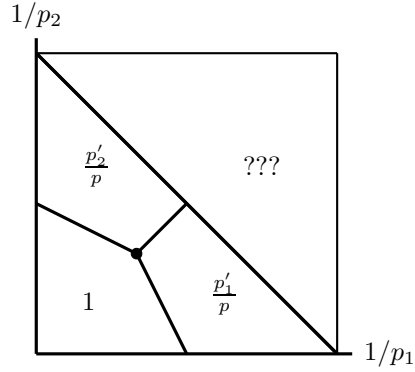

In order to prove Theorem \ref{thm:main} we will approximate multilinear Calder\'on-Zygmund operators by positive dyadic operators.  The result of Dami\'{a}n, Lerner and P\'{e}rez \cite{DLP} states the following (see Section \ref{sect2} for pertinent definition).
\begin{Theorem}\cite[Theorem 1.4]{DLP}\label{sparsebd}
Let $T$ be a multilinear Calder\'{o}n-Zygmund operator and let $\mathcal{X}$ be
a Banach function space over $\bbR^n$ equipped with Lebesgue measure. Then, for any
appropriate $\vec{f}$,
\begin{equation}\label{eq:tlesaX}
 \|T(\vec{f})\|_{\mathcal{X}}\le C_{T,m,n}\sup_{\mathscr{D},\mathcal{S}}\|A_{\mathscr{D},\mathcal{S}}(|\vec{f}|)\|_{\mathcal{X}}.
\end{equation}
\end{Theorem}

When $p\geq 1$,  $\mathcal{X}=L^p(v)$ is a Banach space.  However, for $0<p<1$ it is not.  Since the $m$-linear operators map into $L^p$ for $p>1/m$ we are are unable to obtain the full range.  It is an interesting question as to whether the same decomposition can be obtained for non Banach spaces such as $L^p$ for $0<p<1$.  Moreover, we believe that inequality \eqref{eq:eT} should hold for all $1<p_1,\ldots,p_m<\infty$.

We will actually prove the estimate in Theorem \ref{thm:main} for the sparse operators $A_{\D,\Sp}$.  For these operators, the main techniques are an extension of those found in \cite{M}, in which  the second author proved the sharp weighted bound for linear Calder\'on-Zygmund operators without extrapolation.

The rest of this article is devoted to the following.  In Section \ref{sect2} we state some brief preliminary material.  In Section \ref{sect4} we will prove the main estimates in Theorems \ref{thm:m} and \ref{thm:main} and in Section \ref{sect5} we will provide examples to show that our results are sharp.

\section{Preliminaries}\label{sect2}
Recall that the standard dyadic grid in $\bbR^n$ consists of the cubes
\[
  2^{-k}([0,1)^n+j),\quad k\in\bbZ, j\in\bbZ^n.
\]
Denote the standard grid by $\mathcal D$.

By a general dyadic grid $\mathscr{D}$ we mean a collection of cubes with the following
properties: (i) for any $Q\in\mathscr{D}$ its sidelength, $l_Q$, is of the form $2^k$, $k\in\bbZ$; (ii)
 $Q\cap R \in \{Q,R,\emptyset\}$ for any $Q,R\in\mathscr{D}$; (iii) the cubes of a fixed sidelength $2^k$ form a partition
of $\bbR^n$.

Now we define the dyadic maximal function with respect to arbitrary weight
 \[
   M_{w}^{\mathscr{D}}f(x)=\sup_{Q\ni x, Q\in \mathscr{D}}\frac{1}{w(Q)}\int_Q |f|w.
 \]
 It is well-known that
 \begin{equation}\label{eq:em:p}
 \|M_{w}^{\mathscr{D}}f\|_{L^p(w)}\le p'\|f\|_{L^p(w)},\quad\mbox{$1<p<\infty$},
 \end{equation}
we refer the readers to \cite{M} for a proof.

We will also need the notion of a sparse family of cubes.  Given a dyadic grid $\mathscr{D}$ we say that a family $\Sp$ is sparse if there are disjoint majorizing subsets, that is, for each $Q\in \Sp$ there exists $E_Q\subset Q$ such that $\{E_Q\}_{Q\in \Sp}$ is pairwise disjoint and $|E_Q|\geq \frac12|Q|$.  Sparse families have long played a role in Calder\'on-Zygmund theory, our definition can be found in \cite{Hy2}.  Finally, in \cite{DLP}, the multilinear sparse operators, 
\[
  A_{\mathscr{D},\mathcal{S}}(\vec{f})=\sum_{Q\in \Sp}\bigg( \prod_{i=1}^m \frac{1}{|Q|}\int_{Q}f_i\bigg)\chi^{}_{Q},
\]
were defined and used to approximate multlinear Calder\'on-Zygmund operators (see Theorem \ref{sparsebd}).  


\section{Proof of Theorems \ref{thm:m} and \ref{thm:main}} \label{sect4}
First, we give a proof for Theorem~\ref{thm:m}.
\begin{proof}[Proof of Theorem~\ref{thm:m}]
In \cite{DLP}, the authors proved that
there exists $2^n$ families of dyadic grids $\mathscr{D}_\beta$ such that
\[
  \mathcal{M}(\vec{f})(x)\le 6^{mn}\sum_{\beta=1}^{2^n}\mathcal{M}^{\mathscr{D}_\beta}(\vec{f})(x),
\]
where
\[
    \mathcal{M}^{\mathscr{D}_\beta}(\vec{f})(x)=\sup_{Q\ni x, Q\in\mathscr{D}_\beta
    }\prod_{i=1}^m\frac{1}{|Q|}\int_{Q}|f_i|.
\]

Without loss of generality, it suffices to prove that
\[
  \|\mathcal{M}^{\mathscr{D}}(\vec{f\sigma})\|_{L^p(v_{\vec{w}})}\le C_{m,n,\vec{P}} [\vec{w}]_{A_{\vec{P}}}^{\max_i(\frac{p_i}{p})}\prod_{i=1}^m \|f_i\|_{L^{p_i}(\sigma_i)}.
\]
for a general dyadic grid $\mathscr{D}$, and $\M^\D(\vec{f\sigma})=\M^\D(f_1\sigma_1,\ldots,f_m\sigma_m)$.  Moreover, it was shown in \cite[Lemma 2.2]{DLP} that there exists a sparse subset $\Sp\subset \mathscr D$ such that
$$\mathcal{M}^{\mathscr{D}}(\vec{f\sigma})\lesssim \sum_{Q\in \Sp}\prod_{i=1}^m\Big(\frac{1}{|Q|}\int_Q |f_i|\sigma_i\,\Big)\chi_{E_Q}.$$

Without loss of generality, assume that $p_1=\min\{p_1,\cdots,p_m\}$. We have
\begin{eqnarray*}
\lefteqn{\int_{\bbR^n}\mathcal{M}^{\mathscr{D}}(\vec{f\sigma})^pv_{\vec{w}} \lesssim \sum_{Q\in \Sp} \prod_{i=1}^m\Big(\frac{1}{|Q|}\int_Q |f_i|\sigma_i\,\Big)^pv_{\vec{w}}(Q)}\\
&=&\sum_{Q\in \Sp}\frac{v_{\vec{w}}(Q)^{p_1'}\prod_{i=1}^m\sigma_i(Q)^{pp_1'/{p_i'}}}{|Q|^{mpp_1'}}\bigg( \prod_{i=1}^m \int_{Q}|f_i|\sigma_i \bigg)^p\nonumber\\
&&\quad\cdot \frac{|Q|^{mp(p_1'-1)}}{v_{\vec{w}}(Q)^{p_1'-1}\prod_{i=1}^m\sigma_i(Q)^{pp_1'/{p_i'}}}\nonumber\\
&\leq& [\vec{w}]_{A_{\vec{P}}}^{p_1'}\sum_{Q\in \Sp}\frac{2^{mp(p_1'-1)}|E_Q|^{mp(p_1'-1)}}{v_{\vec{w}}(Q)^{p_1'-1}\prod_{i=1}^m\sigma_i(Q)^{pp_1'/{p_i'}}}\cdot \bigg( \prod_{i=1}^m \int_{Q}|f_i|\sigma_i \bigg)^p.\nonumber
\end{eqnarray*}
By H\"{o}lder's inequality, we have
\begin{eqnarray}
|E_Q|&=&\int_{E_Q}v_{\vec{w}}^{\frac{1}{mp}}\sigma_1^{\frac{1}{mp_1'}}\cdots\sigma_m^{\frac{1}{mp_m'}}\label{eq:h} \\
&\le& v_{\vec{w}}(E_Q)^{\frac{1}{mp}}\sigma_1(E_Q)^{\frac{1}{mp_1'}}\cdots\sigma_m(E_Q)^{\frac{1}{mp_m'}}\nonumber.
\end{eqnarray}
Therefore,
\[
  |E_Q|^{mp(p_1'-1)}\le v_{\vec{w}}(E_Q)^{p_1'-1}\sigma_1(E_Q)^{\frac{p(p_1'-1)}{p_1'}}\cdots\sigma_m(E_Q)^{\frac{p(p_1'-1)}{p_m'}}
\]
and
\[
  \frac{p(p_1'-1)}{p_i'}-\frac{p}{p_i}=\frac{pp_1'}{p_i'}-p\ge 0.
\]
Since $E_Q\subset Q$, we have
\[
  v_{\vec{w}}(E_Q)^{p_1'-1}\le v_{\vec{w}}(Q)^{p_1'-1}
\]
and hence
\[
  \sigma_i(E_Q)^{\frac{p(p_1'-1)}{p_i'}-\frac{p}{p_i}}\le  \sigma_i(Q)^{\frac{pp_1'}{p_i'}-p},\quad i=1,\cdots,m.
\]
It follows that
\begin{eqnarray*}
\lefteqn{\sum_{Q\in \Sp}\frac{|E_Q|^{mp(p_1'-1)}}{v_{\vec{w}}(Q)^{p_1'-1}\prod_{i=1}^m\sigma_i(Q)^{pp_1'/{p_i'}}}
\cdot \bigg( \prod_{i=1}^m \int_{Q}|f_i|\sigma_i\bigg)^p}\\
&\le&\sum_{Q\in \Sp}\prod_{i=1}^m
\bigg(  \frac{1}{\sigma_i(Q)}\int_{Q}|f_i|\sigma_i\bigg)^p \sigma_i(E_Q)^{p/{p_i}}\nonumber\\
&\le&\prod_{i=1}^m\left( \sum_{Q\in \Sp}\bigg(\frac{1}{\sigma_i(Q)}\int_{Q}|f_i|\sigma_i \bigg)^{p_i} \sigma_i(E_{Q})\right)^{p/{p_i}}\nonumber\\
&\le&\prod_{i=1}^m\|M_{\sigma_i}^{\D}(f_i)\|_{L^{p_i}(\sigma_i)}^p\nonumber\\
&\lesssim&\prod_{i=1}^m\|f_i\|_{L^{p_i}(\sigma_i)}^p.\nonumber
\end{eqnarray*}
Hence
\[
  \|\mathcal{M}^{\D}(\vec{f})\|_{L^p(v_{\vec{w}})}\le C_{m,n,\vec{P}} [\vec{w}]_{A_{\vec{P}}}^{\max_i(\frac{p_i'}{p})}\prod_{i=1}^m \|f_i\|_{L^{p_i}(w_i)}.
\]
This completes the proof.
\end{proof}

We now turn our attention to the proof of Theorem~\ref{thm:main}.  First, we note the following symmetry of $A_{\vec{P}}$ weights.
\begin{Lemma}\label{lm:w}
Suppose that $\vec{w}=(w_1,\cdots,w_m)\in A_{\vec{P}}$ and that $1<p$, $p_1$, $\cdots$, $p_m<\infty$ with $1/{p_1}+\cdots+1/{p_m}=1/p$.
Then $\vec{w}^i:=(w_1$, $\cdots$, $w_{i-1}$, $v_{\vec{w}}^{1-p'}$, $w_{i+1}$, $\cdots$, $w_m)\in A_{\vec{P}^i}$ with $\vec{P}^i=(p_1,\cdots, p_{i-1}, p', p_{i+1},\cdots,p_m)$ and
\[
 [\vec{w}^i]_{A_{\vec{P}^i}}=[\vec{w}]_{A_{\vec{P}}}^{p_i'/p}.
\]
\end{Lemma}
\begin{proof}
We will prove the conclusion for $i=1$; the other cases are similar.
Notice that
\[
  1/{p'}+1/{p_{2}}+\cdots+1/{p_m}=1/{p_1'}
\]
and
\[
v_{\vec{w}}^{(1-p')p_1'/{p'}}\cdot w_{2}^{p_1'/{p_{2}}}\cdots w_m^{p_1'/p_m}=w_1^{1-p_1'}.
\]
By the definition of multiple $A_{\vec{P}}$ constant, we have
\begin{eqnarray*}
[\vec{w}^1]_{A_{\vec{P}^1}}&=&\sup_Q \left(\frac{1}{|Q|}\int_Q w_1^{1-p_1'}\right)\cdot\left(\frac{1}{|Q|}\int_Q (v_{\vec{w}}^{1-p'})^{1-p}\right)^{p_1'/p}\\
&&\qquad \times
\prod_{i=2}^m\left( \frac{1}{|Q|}\int_Q w_i^{1-p_i'}\right)^{p_1'/{p_i'}}\\
&=&[\vec{w}]_{A_{\vec{P}}}^{p_1'/p}.
\end{eqnarray*}
\end{proof}
By Theorem \ref{sparsebd} we reduce our problem to consider the behavior of the operator $A_{\mathscr{D},\mathcal{S}}$.  For these operators we have the following Theorem, which holds for all $1<p_1,\ldots,p_m<\infty$.

\begin{Theorem}\label{lm:l1}
Suppose that $1<p_1,\cdots,p_m$ $<\infty$ with $1/{p_1}+\cdots+1/{p_m}=1/p$ and $\vec{w}\in A_{\vec{P}}$. Then
\[
  \|A_{\mathscr{D},\mathcal{S}}(\vec{f})\|_{L^p(v_{\vec{w}})}\lesssim [\vec{w}]^{\max(1,\frac{p_1'}{p},\ldots,\frac{p_m'}{p})}_{A_{\vec{P}}}\prod_{i=1}^m\|f_i\|_{L^{p_i}(w_i)}.
\]
\end{Theorem}
\begin{proof} We first consider the case when $\frac1m<p\le1$.  In this case
$$\int_{\bbR^n}A_{\D,\Sp}(\vec{f})^pv_{\vec{w}}\leq\sum_{Q\in \Sp}\Big(\prod_{i=1}^m\frac{1}{|Q|}\int_Q f_i\Big)^pv_{\vec{w}}(Q),$$
which can be handled in exactly the same manner as the estimates in proof of Theorem \ref{thm:m}.

Now consider the case $p\geq \max_i p_i'$.  It is sufficient to prove that
\[
  \|A_{\mathscr{D},\mathcal{S}}(\vec{f\sigma})\|_{L^p(v_{\vec{w}})}\lesssim [\vec{w}]_{A_{\vec{P}}}\prod_{i=1}^m\|f_i\|_{L^{p_i}(\sigma_i)},
\]
where $\sigma_i=w_i^{1-p_i'}$, $A_{\mathscr{D},\mathcal{S}}(\vec{f\sigma})= A_{\mathscr{D},\mathcal{S}}(f_1\sigma_1,\cdots, f_m\sigma_m)$, and $f_i\geq 0$.
By duality, it suffices to estimate the $(m+1)$-linear form
\begin{equation*}
  \int_{\bbR^n} A_{\mathscr{D},\mathcal{S}}(\vec{f\sigma})g v_{\vec{w}}=
\sum_{Q\in \Sp} \int_{Q} g v_{\vec{w}}\cdot \prod_{i=1}^m \frac{1}{|Q|}\int_{Q}f_i\sigma_i
\end{equation*}
for $g\ge 0$ belonging to $L^{p'}(v_{\vec{w}})$.
We have
\begin{eqnarray*}
\lefteqn{\sum_{Q\in \Sp} \int_{Q}g v_{\vec{w}}\cdot\prod_{i=1}^m \frac{1}{|Q|}\int_{Q}f_i\sigma_i }\\
&=&\sum_{Q\in\Sp} \frac{v_{\vec{w}}(Q)\prod_{i=1}^m \sigma_i(Q)^{p/{p_i'}}}{|Q|^{mp}}\cdot
\frac{|Q|^{m(p-1)}}{v_{\vec{w}}(Q)\prod_{i=1}^m \sigma_i(Q)^{p/{p_i'}}}\\
&&\quad\cdot \int_{Q}g v_{\vec{w}}\cdot\prod_{i=1}^m \int_{Q}f_i\sigma_i \\
&\le&[\vec{w}]_{A_{\vec{P}}}\sum_{Q\in \Sp}\frac{|Q|^{m(p-1)}}{v_{\vec{w}}(Q)\prod_{i=1}^m \sigma_i(Q)^{p/{p_i'}}}
\cdot \int_{Q}g v_{\vec{w}}\cdot\prod_{i=1}^m \int_{Q}f_i\sigma_i \\
&\le&[\vec{w}]_{A_{\vec{P}}}\sum_{Q\in \Sp}\frac{2^{m(p-1)}|E_Q|^{m(p-1)}}{v_{\vec{w}}(Q)\prod_{i=1}^m \sigma_i(Q)^{p/{p_i'}}}
\cdot \int_{Q}g v_{\vec{w}}\cdot \prod_{i=1}^m \int_{Q}f_i\sigma_i .
\end{eqnarray*}
By $(\ref{eq:h})$,
\[
  |E_Q|\le
 v_{\vec{w}}(E_Q)^{\frac{1}{mp}}\sigma_1(E_Q)^{\frac{1}{mp_1'}}\cdots\sigma_m(E_Q)^{\frac{1}{mp_m'}}.
\]
Since $p\ge \max_i\{p_i'\}$ and $E_Q\subset Q$, we have $\sigma_i(Q)^{1-\frac{p}{p_i'}}\le \sigma_i(E_Q)^{1-\frac{p}{p_i'}}$ for any $i=1,\cdots,m$.
Therefore,
\begin{eqnarray*}
&&\sum_{Q\in \Sp} \int_{Q}g v_{\vec{w}}\cdot \prod_{i=1}^m \frac{1}{|Q|}\int_{Q}f_i\sigma_i \\
&\le& 2^{m(p-1)}[\vec{w}]_{A_{\vec{P}}}\sum_{j,k}v_{\vec{w}}(E_Q)^{\frac{1}{p'}}\prod_{i=1}^m \sigma_i(E_Q)^{\frac{p-1}{p_i'}}\sigma_i(Q)^{1-\frac{p}{p_i'}}
   \\
&&\quad\cdot \frac{1}{v_{\vec{w}}(Q)}\int_{Q}g v_{\vec{w}}\cdot\prod_{i=1}^m \frac{1}{\sigma_i(Q)}\int_{Q}f_i\sigma_i \\
&\le& 2^{m(p-1)}[\vec{w}]_{A_{\vec{P}}}\sum_{Q\in \Sp}v_{\vec{w}}(E_Q)^{\frac{1}{p'}}\prod_{i=1}^m \sigma_i(E_Q)^{\frac{1}{p_i}}
\frac{1}{v_{\vec{w}}(Q)}\int_{Q}g v_{\vec{w}}\\
&&\quad\cdot \prod_{i=1}^m \frac{1}{\sigma_i(Q)}\int_{Q}f_i\sigma_i \\
&\le& 2^{m(p-1)}[\vec{w}]_{A_{\vec{P}}}\left( \sum_{Q\in \Sp}\bigg(\frac{1}{v_{\vec{w}}(Q)}\int_{Q}g v_{\vec{w}}\bigg)^{p'}v_{\vec{w}}(E_Q)\right)^{1/{p'}}\\
&&\quad\cdot \prod_{i=1}^m\left( \sum_{Q\in \Sp}\bigg(\frac{1}{\sigma_i(Q)}\int_{Q}f_i \sigma_i \bigg)^{p_i}\sigma_i(E_Q)\right)^{1/{p_i}}\\
&\le& 2^{m(p-1)}[\vec{w}]_{A_{\vec{P}}}\|M_{v_{\vec{w}}}^{\mathscr{D}}(g)\|_{L^{p'}(v_{\vec{w}})}\prod_{i=1}^m \|M_{\sigma_i}^{\mathscr{D}}(f_i)\|_{L^{p_i}(\sigma_i)}\\
&\lesssim& 2^{m(p-1)}[\vec{w}]_{A_{\vec{P}}}\|g\|_{L^{p'}(v_{\vec{w}})}\prod_{i=1}^m \|f_i\|_{L^{p_i}(\sigma_i)},
\end{eqnarray*}
where $(\ref{eq:em:p})$ is used in the last step.  For the other cases we use duality.  Notice that the operator $A_{\D,\Sp}$ is self adjoint as a multilinear operator, in the sense that for any $i$, $i=1,\ldots,m$, we have
$$\int_{\bbR^n} A_{\D,\Sp}(f_1,\ldots,f_m)g= \int_{\bbR^n} A_{\D,\Sp}(f_1,\ldots,f_{i-1},g,f_{i+1},\ldots f_m)f_i.$$
Without loss of generality suppose $p_1'\geq \max(p,p_2',\ldots,p_m')$.  Hence, by duality and self adjiontness we have
\begin{align*}\|A_{\D,\Sp}\|_{L^{p_1}(w_1)\times\cdots\times L^{p_m}(w_m)\rightarrow L^p(v_{\vec{w}})}&=\|A_{\D,\Sp}\|_{L^{p'}(v_{\vec{w}}^{1-p'})\times\cdots\times L^{p_m}(w_m)\rightarrow L^{p_1'}(w_1^{1-p_1'})}\\
&\lesssim [\vec{w}^1]_{\vec{P}^1}= [\vec{w}]_{A_{\vec{P}}}^{\frac{p_1'}{p}}.
\end{align*}
\end{proof}

\section{Examples} \label{sect5}

Finally, we end with some examples to show that our bounds are sharp.  First we show that Theorem \ref{thm:m} is sharp.  Consider the case $m=2$ (we leave it to the reader to modify the example for $m>2$) and suppose that we had a better exponent than the one in inequality \ref{eq:emm}, that is, suppose
\begin{equation}\label{sharp}\|\M\|_{L^{p_1}(w_1)\times L^{p_2}(w_2) \rightarrow L^p(v_{\vec{w}})}\lesssim [\vec{w}]_{A_{\vec{P}}}^{r\max(\frac{p_1'}{p},\frac{p_2'}{p})}\end{equation}
for some $r<1$.  Further suppose that $p_1'\geq p_2'$.  For $0<\ep<1$, let $f_1(x)=|x|^{\ep-n}\chi^{}_{B(0,1)}(x)$, $f_2(x)=|x|^{\frac{\ep-n}{p_2}}\chi_{B(0,1)}(x)$, $w_1(x)=|x|^{(n-\ep)(p_1-1)}$ and $w_2(x)=1$.  Calculations show that
$$\|f_1\|_{L^{p_1}(w_1)}\simeq\ep^{-{1}/{p_1}}, \|f_2\|_{L^{p_2}(w_2)}\simeq\ep^{-1/p_2}, v_{\vec{w}}(x)=|x|^{(n-\ep)\frac{p}{p_1'}}$$
and
$$[\vec{w}]_{A_{\vec{P}}}\simeq\ep^{-{p}/{p_1'}}.$$
For $x\in B(0,1)$ we have
\begin{align*}
\M(f_1,f_2)(x)&\gtrsim\frac{1}{|x|^n}\int_{B(0,|x|)} |y_1|^{\ep-n}\,dy_1 \cdot \frac{1}{|x|^n}\int_{B(0,|x|)} |y_2|^{\frac{\ep-n}{p_2}}\,dy_2\\
&\gtrsim\frac{f_1(x)f_2(x)}{\ep\cdot(\frac{\ep-n}{p_2}+n)}\gtrsim \frac{f_1(x)f_2(x)}{\ep}.
\end{align*}
Hence,
\begin{align*}\|\M(f_1,f_2)\|_{L^p(v_{\vec{w}})}&\gtrsim \frac{1}{\ep}\Big(\int_{B(0,1)}|x|^{(\ep-n)(p+\frac{p}{p_2}-\frac{p}{p_1'})}\,dx\Big)^{1/p}\\
&\simeq\frac1\ep\Big(\int_0^1x^{\ep-1}\,dx\Big)^{1/p}\\
&=\frac{1}{\ep}\Big(\frac{1}{\ep}\Big)^{1/p}.
\end{align*}
Combining this with inequality \eqref{sharp} we see for some $r<1$,
$$\Big(\frac1\ep\Big)^{1+\frac1p}\lesssim \Big(\frac1\ep\Big)^{r+\frac1p},$$
which is impossible as $\ep\rightarrow 0$.

 Next we show that Theorem \ref{thm:main} is sharp.
  Recall that for $i=1,\cdots,n$, the $m$-linear $i$th Riesz transform is defined by
\[
  R_i(\vec{f})(x)=p.v.\int_{(\bbR^n)^m}\frac{\sum_{j=1}^m(x_i-(y_j)_i)}{(\sum_{j=1}^m|x-y_j|^2)^{(nm+1)/2}}f_1(y_1)\cdots f_m(y_m)dy_1\cdots dy_m,
\]
where $(y_j)_i$ denotes the $i$th coordinate of $y_j$.

Suppose that $m=2$, $p_1'\ge p_2'$ and $p_1'\ge p$.
Let
\begin{eqnarray*}
  U&=&\{x\in\bbR^n: |x|\le 1,0<x_i\le x_1,i=2,\cdots,n\}, \\
   V &=& \{x\in\bbR^n:\, |x|\le 1, x_i\le 0, i=1,\cdots,n \}.
\end{eqnarray*}
For $0<\ep<1$, let $f_1(x)=|x|^{\ep-n}\chi^{}_V(x)$, $f_2(x)=|x|^{\frac{\ep-n}{p_2}}\chi^{}_V(x)$, $w_1(x)=|x|^{(n-\ep)(p_1-1)}$ and $w_2(x)=1$.  For $x\in U$ and $y_j\in V$ with $|y_j|\le |x|$, we have
\[
  \frac{\sum_{j=1}^2(x_1-(y_j)_1)}{(\sum_{j=1}^2|x-y_j|^2)^{1/2}}
  \ge\frac{2\frac{|x|}{\sqrt{n}}}{4|x|}\gtrsim 1.
\]
Therefore,
\[
    \frac{\sum_{j=1}^2(x_1-(y_j)_1)}{(\sum_{j=1}^2|x-y_j|^2)^{(2n+1)/2}}\gtrsim \frac{1}{|x|^{2n}}.
\]
It follows that
\begin{eqnarray*}
R_1(\vec{f})(x)&=&p.v.\int_{(\bbR^n)^2}\frac{\sum_{j=1}^2(x_i-(y_j)_i)}{(\sum_{j=1}^2|x-y_j|^2)^{(2n+1)/2}}f_1(y_1) f_2(y_2)dy_1dy_2\\
&\gtrsim&\int_{|y_1|\le |x|}\int_{|y_2|\le |x|}\frac{1}{|x|^{2n}}
|y_1|^{\ep-n}\cdot|y_2|^{\frac{\ep-n}{p_2}}dy_1dy_2\\
&\gtrsim& \frac{1}{\ep}f_1(x)f_2(x).
\end{eqnarray*}
Hence
\begin{eqnarray*}
\|R_1(\vec{f})\|_{L^p(v_{\vec{w}})}&\gtrsim& \frac{1}{\ep}\left(\int_U |x|^{(\ep-n)(p+p/{p_2}-p/{p_1'})}dx\right)^{1/p}\\
&=& \frac{1}{\ep}\left(\int_U |x|^{\ep-n}dx\right)^{1/p}\\
&\gtrsim& \frac{1}{\ep}\left(\int_{\{|x|\le 1\}} |x|^{\ep-n}dx\right)^{1/p}\quad\mbox{(by symmetry)}\\
&\gtrsim& (\frac{1}{\ep})^{1+1/p}.
\end{eqnarray*}
Then by similar arguments as the above we can show that the exponent is sharp when $\max( p_1',p_2')\ge p\ge 1$.  When $p>\max(p_1',p_2')$. Again, suppose that $p_1'\ge p_2'$. We consider the adjoint in the first variable, $(R_1)^{1,*}$. Notice that
\[
  (R_1)^{1,*}(f_1,f_2)(x)=\int_{(\bbR^n)^2}\frac{2(y_1)_1-x_1-(y_2)_1}{(|x-y_1|^2+|y_1-y_2|^2)^{(2n+1)/2}}
  f(y_1)f(y_2)dy_1dy_2.
\]
Let
\begin{eqnarray*}
  U_1&=&\{x\in\bbR^n: |x|\le 1,x_1\le x_i< 0 ,i=2,\cdots,n\}, \\
   V_1 &=& \{x\in\bbR^n:\, |x|\le 1, x_i\ge 0, i=1,\cdots,n \}.
\end{eqnarray*}
For $0<\ep<1$, let $f_1(x)=|x|^{\ep-n}\chi^{}_{V_1}(x)$, $f_2(x)=|x|^{\frac{\ep-n}{p_2}}\chi^{}_{V}(x)$,
$w_1(x)=|x|^{(\ep-n)p_1/p}$ and $w_2(x)=1$. Then $v_{\vec{w}}(x)=|x|^{\ep-n}$, $v_{\vec{w}}^{1-p'}=|x|^{(n-\ep)(p'-1)}$ and $w_1^{1-p_1'}=|x|^{(n-\ep)p_1'/p}$. Similar arguments as the above show that
\begin{eqnarray*}
\|(R_1)^{1,*}\|_{L^{p'}(v_{\vec w}^{1-p'})\times L^{p_2}(w_2)\rightarrow L^{p_1'}(w_1^{1-p_1'})}
\gtrsim \frac{1}{\ep}=[\vec{w}^1]_{A_{\vec{P}^1}}^{p/{p_1'}}=[\vec{w}]_{A_{\vec{P}}}.
\end{eqnarray*}
Therefore,
\[
  \| R_1 \|_{L^{p_1}(w_1)\times L^{p_2}(w_2)\rightarrow L^p(v_{\vec{w}})}=\|(R_1)^{1,*}\|_{L^{p'}(v_{\vec w}^{1-p'})\times L^{p_2}(w_2)\rightarrow L^{p_1'}(w_1^{1-p_1'})}\gtrsim [\vec{w}]_{A_{\vec{P}}}.
\]
This shows the sharpness of the exponent when $p>\max_i \{p_i'\}$, which completes the proof.

\textbf{Acknowledgements}.\,\, The authors thank Carlos P\'erez for helping improve the quality of this article.

\end{document}